\documentclass{article}
\usepackage{graphicx} 
\usepackage{latexsym}
\usepackage{amsfonts, mathrsfs, amsmath}
\usepackage{a4wide}
\usepackage{url}
\usepackage{graphicx}
\usepackage{enumerate}
\usepackage{bm}
\usepackage{amsthm}
\usepackage{color}
\usepackage[colorlinks,citecolor=blue,urlcolor=blue]{hyperref}

\newcommand\Prob{{\mathbb P}}
\newcommand\E{{\mathbb E}}
\newcommand\V{{\rm Var}}
\newcommand\cov{{\rm Cov}}
\newcommand\given{\, \vert \, }

\newtheorem{theorem}{Theorem}

\newtheorem{lemma}{Lemma}

\newtheorem{proposition}{Proposition}

\newcommand\inprob{\buildrel {\rm P} \over  \longrightarrow}
\newcommand\convlaw{\buildrel {\rm D} \over \longrightarrow }

\title{Average Jaccard Index of Random Graphs}
\author{Qunqiang Feng, Shuai Guo and Zhishui Hu{\thanks{Email: huzs@ustc.edu.cn}}\\
{\small Department of Statistics and Finance, School of Management} \\
{\small University of Science and Technology of China}\\
{\small Hefei 230026, China}}
\date{}

\begin{document}
\maketitle

\begin{abstract}
  The asymptotic behavior of the Jaccard index in $G(n,p)$,
  the classical Erd\"{o}s-R\'{e}nyi random graphs model, 
  is studied in this paper, as $n$ goes to infinity.
  We first derive the asymptotic distribution of the Jaccard index 
  of any pair of distinct vertices, 
  as well as the first two moments of this index.  
  Then the average of the Jaccard indices over all vertex pairs
  in $G(n,p)$ is shown to be asymptotically normal under an additional mild condition
  that $np\to\infty$ and $n^2(1-p)\to\infty$. 
  

\end{abstract}

\section{Introduction}

The {\em Jaccard index}, also known as the {\em Jaccard similarity coefficient},
was originally introduced by Paul Jaccard to measure the similarity between two sets~\cite{jaccard1901distribution}. 
For any two finite sets $\mathcal{A}$ and $\mathcal{B}$, 
the Jaccard index $J(\mathcal{A},\mathcal{B})$ is the ratio of the size of their intersection 
to the size of their union.
That is,
\begin{align*}
 J(\mathcal{A},\mathcal{B})=\frac{|\mathcal{A}\cap \mathcal{B}|}{|\mathcal{A}\cup \mathcal{B}|}=\frac{|\mathcal{A}\cap \mathcal{B}|}{|\mathcal{A}|+|\mathcal{B}|-|\mathcal{A}\cap \mathcal{B}|},   
\end{align*}
where the symbol $|\cdot|$ denotes the cardinality of a set. 
It is clear that this index ranges from 0 to 1.
The associated {\em Jaccard distance} 
for quantifying the dissimilarity between two sets
is defined as one minus the Jaccard index (see, e.g., \cite{gilbert1972distance,kosub2019note}). 
In statistics and data science, the Jaccard index is employed as a statistic to 
measure the similarity between 
sample sets, especially for binary and categorical data 
(see, e.g., \cite{chung2019jaccard,koeneman2022improved}).
For extensive generalizations of the Jaccard index in many other mathematical structures, 
such as scalars, vectors, matrices and multisets,
we refer the readers to~\cite{costa2021further}. 
Due to simplicity and popularity, many applications of the Jaccard index 
and its variants have been founded in various fields, 
such as cell formation~\cite{yin2006similarity}, pattern recognition~\cite{henning2007cluster},
data mining~\cite{singh2009privacy}, natural language processing~\cite{wu2017extractin},
recommendation system~\cite{bag2019efficient}, medical
image segmentation~\cite{eelbode2020optim}, and machine learning~\cite{ali2021machine}. 

Following the original definition on the sets,
the Jaccard index of two vertices in a graph can naturally be extended to equal 
the number of common neighbors divided by the number of vertices 
that are neighbors of at least one of them
(see, e.g., \cite{fan2019similarity}).
As a graph benchmark suitable for real-world applications, 
the Jaccard index has also been proposed to determine the similarity 
in graphs or networks~\cite{kogge2016jaccard},
because of its clear interpretability and computational scalability. 
This index, as well as its variants, is employed  to find core nodes 
for community detection in complex networks~\cite{berahmand2018community,miasnikof2022empirical},
to estimate the coupling strength between temporal graphs~\cite{mammone2018permutation},
to do link prediction~\cite{zhang2016measuring,sathre2022edge,lu2023embedding}, and among others.

Erd\"{o}s-R\'{e}nyi random graph is widely used as a benchmark model 
in statistical network analysis (see, e.g., \cite{arias2014community,verzelen2015community}).
In the simulation study of~\cite{shestopaloff2022statistical},  it is shown that
the empirical cumulative distribution functions of Jaccard indices over all vertex pairs in two network models, 
Erd\"{o}s-R\'{e}nyi random graph and the stochastic block model, are quite different.
Despite the widespread applications of the Jaccard index in network analysis, 
to the best of our knowledge, there is lack of comprehensive study of theoretical results 
on this simple index defined on statistical graph models. 
As the first step to fill this gap, 
our main concern in this paper is derive the asymptotic behavior of the basic Jaccard index in Erd\"{o}s-R\'{e}nyi random graphs.
For numerous probabilistic results on this classical random graph model,
we refer the readers to the monographs~\cite{janson2000random},~\cite{bollobas2001random},
and  \cite{hofstad2016random}. 

Throughout this paper, we shall use the notation as follows.
For any integer $n\ge2$, we denote by $[n]$ the vertex set $\{1,2,\cdots,n\}$.
For an event ${\cal E}$, let $|{\cal E}|$ be the cardinality, $\overline{\cal E}$ the complement,
and ${\bm I}({\cal E})$ the indicator of ${\cal E}$. 
For real numbers $a,b$, we write $a\vee b$ to denote the maximum of $a$ and $b$.
For probabilistic convergence, 
we use $\convlaw$ and $\inprob$ to denote convergence in distribution
and in probability, respectively.

The rest of this paper is organized as follow. 
The Jaccard index of any pair of distinct vertices in the Erd\"{o}s-R\'{e}nyi random graph
$G(n,p)$ is considered in Section 2.
We first compute the mean and variance of this index, 
and then obtain the phase changes of its asymptotic distribution for $p\in [0,1]$ in all regimes, as $n\to \infty$. 
In Section 3, we prove the asymptotic normality of the average of the Jaccard indices over 
all vertex pairs in $G(n,p)$, as $np\to \infty$ and $n^2(1-p)\to \infty$. 

\section{Jaccard index of a vertex pair}
Let us denote by $G(n,p)$ an Erd\"{o}s-R\'{e}nyi random graph on the vertex set $[n]$, 
where each edge is present independently with probability $p$. 
In this paper, we shall consider $p = p(n)$ as a function of the graph size $n$. 
For any two vertices $i, j \in [n]$, let $I_{ij}$ be the indicator
that takes the value 1 if an edge between $i$ and $j$ is present in $G(n,p)$, and takes the value 0 otherwise. 
 It follows that $I_{ii} = 0$, $I_{ij} = I_{ji}$, and $\{I_{ij} : 1 \leq i < j \leq n\}$ 
 is a sequence of independent Bernoulli variables with success rate $p$. 
 The $n\times n$ matrix ${\bm A}=(I_{ij})$ is usually called the adjacent matrix of $G(n,p)$,
 and is a symmetric matrix with all diagonal entries equal to zero.
 
 For any vertex $i\in[n]$, let the set ${\cal N}_i$ be its neighborhood, i.e.,  
 ${\cal N}_{i}=\{k:I_{ik}=1,  k\in [n]\}$. 
 For any pair of vertices $i,j\in[n]$,
 we also define their union neighborhood as 
 \begin{align*}
 {\cal N}_{ij}=\big\{k:I_{ik}\vee I_{jk}=1, 
 ~k\in[n]~\mbox{and}~k\neq i,j\big\}, \quad i\neq j.
 \end{align*}
 Notice that here the neighborhood set ${\cal N}_{ij}$ does not contain vertices 
 $i$ and $j$ themselves, even if $I_{ij}=1$.
 Then the Jaccard index of vertices $i$ and $j$ in $G(n,p)$ is formally defined as 
\begin{align}\label{defineJ}
	J_{ij}^{(n)}=\frac{|{\cal N}_{i}\cap {\cal N}_{j}|}{|{\cal N}_{ij}|}
=:\frac{S_{ij}^{(n)}}{T_{ij}^{(n)}},\quad i\neq j.
\end{align}
One may see that the index $J_{ij}^{(n)}$
given in \eqref{defineJ} is not well-defined 
when ${\cal N}_{ij}$ is a empty set or $T_{ij}^{(n)}=0$.
For convenience, following the idea in \cite{chung2019jaccard}
we shall define $J_{ij}^{(n)}=p/(2-p)$ in this special case.
Indeed, it is shown later that the conditional expectation of $J_{ij}^{(n)}$ is exactly $p/(2-p)$ 
given that $T_{ij}^{(n)}>0$.
In terms of the adjacent matrix ${\bm A}$,  the numerator and denominator in \eqref{defineJ} 
can be rewritten as 
\begin{align}\label{sijntijn}
    S_{ij}^{(n)}=\sum\limits_{k\neq i,j} I_{ik}I_{jk},\quad T_{ij}^{(n)}=\sum\limits_{k\neq i,j} I_{ik}\vee I_{jk}.
\end{align}
Due to independence of the elements in ${\bm A}$, it is clear that the random variables
$S_{ij}^{(n)}$ and $T_{ij}^{(n)}$ follow the binomial distributions 
${\rm Bin}(n-2,p^2)$ and ${\rm Bin}(n-2,p(2-p))$, 
respectively. Hence, the Jaccard index of a vertex pair in $G(n,p)$
is a quotient of two dependent binomial random variables. 

\subsection{Mean and variance}
We first calculate the mean and variance of the Jaccard index
of any pair of vertices in $G$. By \eqref{defineJ} and \eqref{sijntijn},
one can see that $\{J_{ij}^{(n)}, 1\le i<j\le n\}$ is a sequence of random variables that are
pairwise dependent but identically distributed.
Without loss of generality, we only consider $J_{12}^{(n)}$.

For any vertex $3\le k\le n$, the conditional probability
\begin{align*}
    \mathbb{P}\left(I_{1k}I_{2k}=1 \given I_{1k}\vee I_{2k}=1\right)=\frac{\mathbb{P}\left(I_{1k}I_{2k}=1\right)}{\mathbb{P}\left(I_{1k}\vee I_{2k}=1\right)}=\frac{p}{2-p},
\end{align*}
which is regardless of $k$.
Then for any positive integer $1\le m\le n-2$, given the event $T_{12}^{(n)}=m$, 
the conditional distribution of $S_{12}^{(n)}$ is ${\rm Bin}(m,p/(2-p))$,
due to independence of the indicators $\{I_{1k},I_{2k},3\le k\le n\}$. 
Consequently, we have
\begin{align}\label{conditionalE}
	\E\left[S_{12}^{(n)}\Big|T_{12}^{(n)}=m\right]=\frac{mp}{2-p},
\end{align}
and
\begin{align}\label{conditionalVar}
	\V\left[S_{12}^{(n)}\Big|T_{12}^{(n)}=m\right]=\frac{2mp(1-p)}{(2-p)^{2}}.
\end{align}
Noting that $J_{12}^{(n)}=p/(2-p)$ in the special case $T_{12}^{(n)}=0$,
by \eqref{defineJ} and \eqref{conditionalE} we thus have
\begin{align*}
    \E\Big[J_{12}^{(n)}\Big|T_{12}^{(n)}\Big]=\frac{p}{2-p},
\end{align*}
which implies that $\E[J_{12}^{(n)}]=p/(2-p)$, and
\begin{align}\label{conditonalEVar}
    \V\left[\E\left(J_{12}^{(n)}\Big|T_{12}^{(n)}\right)\right]=0.
\end{align}
Using the law of total variance, it follows by (\ref{conditionalVar}) and (\ref{conditonalEVar}) that
\begin{align}\label{VarJ12n}
	\V\left[J_{12}^{(n)}\right]&=\E\left[\V\left(J_{12}^{(n)}\Big|T_{12}^{(n)}\right)\right]\notag\\
	&=\sum_{m=1}^{n-2}\Prob\left(T_{12}^{(n)}=m\right)\V\Big[\frac{S_{12}^{(n)}}{m}\Big|T_{12}^{(n)}=m\Big]\notag\\
	&=\frac{2p(1-p)}{(2-p)^{2}}\sum_{m=1}^{n-2}\frac{1}{m}\Prob\big(T_{12}^{(n)}=m\big),
\end{align}
which involves the first inverse moment of the binomial distribution. 
Recalling that $T_{12}^{(n)}$ has the distribution ${\rm Bin}(n-2,p(2-p))$, 
it follows by Corollary 1 in \cite{wang2008asymptotic} that
\begin{align}\label{inverse}
    \sum_{m=1}^{n-2}\frac{1}{m}\Prob\left(T_{12}^{(n)}=m\right)=\frac{1}{np(2-p)}
    \Big(1+O\Big(\frac{1}{np}\Big)\Big),
\end{align}
as $np\to\infty$.
By \eqref{VarJ12n} and \eqref{inverse}, we thus have
\begin{align*}
    \V\Big[J_{12}^{(n)}\Big]=\frac{2(1-p)}{n(2-p)^{3}}\Big(1+O\Big(\frac{1}{np}\Big)\Big).
\end{align*}
	
Collecting the above findings, by Chebyshev's inequality we thus prove the following result.

\begin{proposition}\label{propjapair}
Let $J_{ij}^{(n)}$ be the Jaccard index of any distinct vertices $i,j\in[n]$ in $G(n,p)$. 
Then $\E[J_{ij}^{(n)}]=p/(2-p)$ for all $n\ge 2$.
In particular, as $np\to\infty$, it further follows that
\begin{align*}
    \V\Big[J_{ij}^{(n)}\Big]=\frac{2(1-p)}{n(2-p)^{3}}\Big(1+O\Big(\frac{1}{np}\Big)\Big),
\end{align*}
and $J_{ij}^{(n)}-p/(2-p)$ converges to 0 in probability.
\end{proposition}

\subsection{Asymptotic distribution}
We now establish the asymptotic distribution of the Jaccard index of any vertex pair in $G(n,p)$ 
in the following.
\begin{theorem}\label{asymptoticdis}
Let $J_{ij}^{(n)}$ be the Jaccard distance of any distinct vertices $i,j\in[n]$ in $G(n,p)$. 
\begin{itemize}
    \item[(i)] If $np^{2}(1-p)\to\infty$, then
    \begin{align*}
        \sqrt{\frac{n(2-p)^{3}}{2(1-p)}}\Big(J_{ij}^{(n)}-\frac{p}{2-p}\Big)\convlaw Z,
    \end{align*}
    where $Z$ denotes a standard normal random variable;
    \item[(ii)] If $np^2\to \lambda$ for some constant $\lambda>0$, 
    then $2npJ_{ij}^{(n)}\convlaw {\rm Poi}(\lambda)$, 
    where ${\rm Poi}(\lambda)$ denotes the Poisson distribution with parameter $\lambda$;
    \item[(iii)] If $np^2\to0$, then $npJ_{ij}^{(n)}\inprob 0$.
    \item[(iv)] If $n(1-p)\to c$ for some constant $c>0$, then $n(1-J_{ij}^{(n)})\convlaw {\rm Poi}(2c)$.
    \item[(v)]  If $n(1-p)\to 0$, then $n(1-J_{ij}^{(n)})\inprob 0$. 
\end{itemize}	
\end{theorem}

\begin{proof}
As presented in the previous subsection, it is sufficient to consider a single index $J_{12}^{(n)}$.
To prove (i), we first rewrite 
\begin{align}\label{normalisedJ12n}
\sqrt{\frac{(n-2)(2-p)^{3}}{2(1-p)}}\Big(J_{12}^{(n)}-\frac{p}{2-p}\Big)
&=\sqrt{\frac{(n-2)(2-p)}{2(1-p)}}\cdot\frac{(2-p)S_{12}^{(n)}-pT_{12}^{(n)}}{T_{12}^{(n)}}\notag\\
&=\frac{(2-p)S_{12}^{(n)}-pT_{12}^{(n)}}{\sqrt{2(n-2)p^{2}(1-p)(2-p)}}\cdot\frac{(n-2)p(2-p)}{T_{12}^{(n)}}.
\end{align}
For any distinct vertices $i,j,k\in[n]$, let us denote 
\begin{align}\label{vijk}
V_{ij,k}=(2-p)I_{ik}I_{jk}-p\big(I_{ik}\vee I_{jk}\big).
\end{align}
Then, for any fixed two vertices $i,j\in[n]$, the random variables 
$\{V_{ij,k},~k\in[n], k\neq i,j\}$ are independent and 
identically distributed with common mean 0 
and common variance $2p^{2}(1-p)(2-p)$.
Since it follows by \eqref{sijntijn} that  
\begin{align}\label{2pspt}
(2-p)S_{12}^{(n)}-pT_{12}^{(n)}=\sum_{k=3}^{n} V_{12,k},
\end{align}
a direct application of the Lindeberg-Feller central limit theorem yields that
\begin{align}\label{limitp}
    \frac{(2-p)S_{12}^{(n)}-pT_{12}^{(n)}}{\sqrt{2(n-2)p^{2}(1-p)(2-p)}} \convlaw Z,
\end{align}
whenever $np^2(1-p)\to\infty$.
By Chebyshev's inequality, the fact $T_{12}^{(n)}\sim {\rm Bin}((n-2),p(2-p))$ gives that as $np\to\infty$,
\begin{align}\label{Tlimit}
\frac{T_{12}^{(n)}}{(n-2)p(2-p)}\inprob 1,
\end{align}
which, together with \eqref{normalisedJ12n} and \eqref{limitp}, proves (i) by Slutsky's lemma.

If $np^{2}\to \lambda$ for some constant $\lambda>0$, we must have that $p\to 0$ and $np\to\infty$. 
Since $S_{12}^{(n)}\sim {\rm Bin}(n-2,p^2)$, 
Poisson limit theorem yields that $S_{12}^{(n)}\convlaw {\rm Poi}(\lambda)$.
By (\ref{Tlimit}), again using Slutsky’s lemma, one can obtain that  
	\begin{align*}
	2npJ_{12}^{(n)}=\frac{2np}{T_{12}^{(n)}}\cdot S_{12}^{(n)}\convlaw \text{Poi}(\lambda),
	\end{align*}
which proves (ii).

If $np^2\to0$, to prove (iii) we only need to prove that the probability $\Prob(npJ_{ij}^{(n)}>np^2/(2-p))$ tends to 0.
Note that if the event $\{S_{12}^{(n)}=0\}$ occurs, the Jaccard index $J_{ij}^{(n)}$ must be 0 or $p/(2-p)$. Therefore,
\begin{align*}
\Prob\Big(npJ_{ij}^{(n)}>\frac{np^2}{2-p}\Big)=\Prob\Big(J_{ij}^{(n)}>\frac{p}{2-p}\Big)\le \Prob\big(S_{12}^{(n)}\neq0\big)
=1-\Prob\big(S_{12}^{(n)}=0\big).
\end{align*}
Again by the fact $S_{12}^{(n)}\sim {\rm Bin}(n-2,p^2)$, we have that $1-\Prob\big(S_{12}^{(n)}=0\big)\to 0$, if $np^2\to0$.  
This implies (iii).  

Analogously to \eqref{normalisedJ12n}, we have
\begin{align}\label{p2pjn}
\frac{p}{2-p}-J_{12}^{(n)}=\frac{-(2-p)S_{12}^{(n)}+pT_{12}^{(n)}}{(n-2)p(2-p)^2}\cdot\frac{(n-2)p(2-p)}{T_{12}^{(n)}}.
\end{align}
Note that \eqref{Tlimit} still holds and $n[1-p/(2-p)]$ has the limit $2c$, if $n(1-p)\to c$ for some constant $c>0$. 
To prove (iv), by \eqref{p2pjn} it is now sufficient to show that 
\begin{align}\label{2plimiv}
-(2-p)S_{12}^{(n)}+pT_{12}^{(n)}\convlaw {\rm Poi}(2c)-2c.
\end{align}
For any distinct vertices $i,j,k\in[n]$, one can directly obtain from \eqref{vijk} that the characteristic function of $-V_{ij,k}$ is 
\begin{align*}
f_n(t)=\E\big[{\rm e}^{-{\rm i}tV_{ij,k}}\big]=p^2{\rm e}^{-2{\rm i}t(1-p)}+2p(1-p){\rm e}^{{\rm i}tp}+(1-p)^2,
\end{align*}
where ${\rm i}=\sqrt{-1}$ denotes the imaginary unit.
Then, by \eqref{2pspt} and independence, we have that the characteristic function of $-(2-p)S_{12}^{(n)}+pT_{12}^{(n)}$ is
equal to
\begin{align*}
f_n^{n-2}(t)&=\big[p^2{\rm e}^{-2{\rm i}t(1-p)}+2p(1-p){\rm e}^{{\rm i}tp}+(1-p)^2\big]^{n-2},\quad  t\in\mathbb{R}.
\end{align*}
Note that 
\begin{align*}
\lim_{n\to\infty}n\big[p^2{\rm e}^{-2{\rm i}t(1-p)}+2p(1-p){\rm e}^{{\rm i}tp}+(1-p)^2-1\big]
&=2c{\rm e}^{{\rm i}t}+\lim_{n\to\infty}n\big[p^2{\rm e}^{-2{\rm i}t(1-p)}-1\big]\\
&=2c{\rm e}^{{\rm i}t}+\lim_{n\to\infty}n\big[p^2\big({\rm e}^{-2{\rm i}t(1-p)}-1)+(p^2-1)\big]\\
&=2c\big({\rm e}^{{\rm i}t}-{\rm i}t-1\big),
\end{align*}
if $n(1-p)\to c$. Therefore, the limit of the characteristic function of $-(2-p)S_{12}^{(n)}+pT_{12}^{(n)}$ satisfies 
\begin{align*}
\lim_{n\to\infty}f_n^{n-2}(t)=\exp\big\{2c\big({\rm e}^{{\rm i}t}-{\rm i}t-1\big)\big\},
\end{align*}
which implies \eqref{2plimiv} and completes the proof of (iv).

We only sketch the proof of (v), since it is very similar to (iv). By \eqref{Tlimit} and \eqref{p2pjn}, 
it is sufficient to show that $-(2-p)S_{12}^{(n)}+pT_{12}^{(n)}$ converges in probability to 0 under the condition $n(1-p)\to 0$.
In fact, following the proof of \eqref{2plimiv}, in this case one can deduce that its characteristic function $f_n^{n-2}(t)\to 1$.  
\end{proof}

\section{Average Jaccard index}
In this section, we derive asymptotic properties of
the average Jaccard index of $G(n,p)$, which is given by
\begin{align}\label{Jndef}
	J_n=\frac{2}{n(n-1)}\sum_{i=1}^{n-1} \sum_{j=i+1}^n J_{ij}^{(n)}.
\end{align}
That is, the average Jaccard index $J_n$ is the average of the Jaccard indices over all vertex pairs in 
the Erd\"{o}s-R\'{e}nyi random graph $G(n,p)$. 
An immediate consequence of Proposition \ref{propjapair} is that the expectation of $J_n$ is equal to $p/(2-p)$. 

We now state our main results of this paper in the following.

\begin{theorem}\label{theomavejac}
Let $J_n$ be the average Jaccard index of $G(n,p)$. 
If $np\to\infty$ and $n^2(1-p)\to\infty$, then
\begin{align*}
\frac{n(2-p)^2}{\sqrt{8p(1-p)}}\Big(J_n-\frac{p}{2-p}\Big)\convlaw Z,    
\end{align*}
where $Z$ denotes a standard normal random variable.
\end{theorem}

It is remarkable that the quantity $n^2(1-p)/2$ is the asymptotic expected number of unoccupied edges in $G(n,p)$.
The above theorem suggests that if $G(n,p)$ is neither too sparse (see, e.g., \cite[Section  1.8]{hofstad2016random}) 
nor close to a complete graph, 
then its average Jaccard index $J_n$ has asymptotic normality.  
In order to prove Theorem \ref{theomavejac}, we first introduce two auxiliary lemmas,
both of which involve the inverse moments of the binomial distribution.

\begin{lemma}\label{leminverse1}
If random variable $X_n$ has a binomial distribution with parameters $n$ and $p$, 
then for any fixed constants $a\ge0$ and $b>0$, as $np\to\infty$,
\begin{align}
\E\Big[\Big(\frac{(n+a)p}{b+X_n}-1\Big)^2\Big]=O\Big(\frac{1}{np}\Big). \label{bino}
\end{align}
\end{lemma}

\begin{proof}
For any $\varepsilon\in (0,1)$, we define
\begin{align*}
{\cal A}_n:=\big\{ (1-\varepsilon)(n+a)p\le b+X_n\le (1+\varepsilon)(n+a)p\big\}.
\end{align*}
Applying the Chernoff's bound for the binomial distribution (see, e.g., \cite[Corollary 2.3]{janson2000random}) gives that
for sufficiently large $n$,
\begin{align*}
  \Prob(\overline{{\cal A}}_n) \le \Prob\Big(\Big|\frac{X_n}{np}-1\Big|>\frac{\varepsilon}{2}\Big)\le
   2\exp\Big\{-\frac{\varepsilon^2}{12} np\Big\}=O\Big(\frac{1}{n^3p^3}\Big).
\end{align*}
Then  for sufficiently large $n$, by noting that
\begin{align*}
 \Big(\frac{(n+a)p}{b+X_n} -1\Big)^2\le \Big(\frac{(n+a)p}{b}\Big)^2+1\le 2n^2p^2,
\end{align*}
we have
\begin{align*}
\E\Big[\Big(\frac{(n+a)p}{b+X_n}-1\Big)^2\Big]&=\E\Big[\Big(\frac{(n+a)p}{b+X_n}-1\Big)^2{\rm I}({\cal A}_n)\Big]+
\E\Big[\Big(\frac{(n+a)p}{b+X_n}-1\Big)^2{\rm I}(\overline{{\cal A}}_n)\Big]\nonumber\\
&\le \E\Big[\frac{(n+a)^2p^2}{(b+X_n)^2}\Big(\frac{b+X_n}{(n+a)p}-1\Big)^2{\rm I}({\cal A}_n)\Big]+2n^2p^2\mathbb{P}(\overline{{\cal A}}_n)\nonumber\\
&\le (1-\varepsilon)^{-2}\E\Big[\Big(\frac{b+X_n}{(n+a)p}-1\Big)^2\Big]+2n^2p^2 \mathbb{P}(\overline{{\cal A}}_n)\nonumber\\
&=O\Big(\frac{1}{np}\Big),
\end{align*}
where we used the inequality
\begin{align*}
  \E\Big[\Big(\frac{b+X_n}{(n+a)p}-1\Big)^2\Big]\le \frac{2}{(n+a)^2p^2} \big(\mathbb{E}[(X_n-np)^2]+(b-ap)^2\big)=O\Big(\frac{1}{np}\Big).  
\end{align*}
The proof of Lemma \ref{leminverse1} is complete.
\end{proof}

\begin{lemma}\label{leminverse2}
If random variable $X_n$ has a binomial distribution with parameters $n$ and $p$, 
then for any fixed positive constants $b$ and $\alpha$, as $np\to\infty$,
\begin{align*}
\E\Big[\frac{1}{(b+X_n)^{\alpha}}\Big]=\frac{1}{(np)^{\alpha}}(1+o(1)). 
\end{align*}
\end{lemma}

\begin{proof}
This is an immediate consequence of Theorem 2 in \cite{shi2010note}.
\end{proof}

We now give a formal proof of Theorem \ref{theomavejac} in the following. 

\begin{proof}[Proof of Theorem \ref{theomavejac}]
By \eqref{p2pjn}, the index $J_{ij}^{(n)}$ can be expressed as 
\begin{align}\label{expand}
    J_{ij}^{(n)}=\frac{p}{2-p}+\frac{(2-p)S_{ij}^{(n)}-pT_{ij}^{(n)}}{(n-2)p(2-p)^2}+R_{ij}^{(n)}, \quad 1\le i\neq j\le n,
\end{align}
where  the remainder term
    \begin{align}\label{res}
        R_{ij}^{(n)}=\frac{(2-p)S_{ij}^{(n)}-pT_{ij}^{(n)}}{(n-2)p(2-p)^2}
        \left(\frac{(n-2)p(2-p)}{T_{ij}^{(n)}}-1\right).
    \end{align}
Note the special case in \eqref{expand} that the remainder term 
$R_{ij}^{(n)}$ vanishes if $T_{ij}^{(n)}=0$. 
Taking expectation on both sides of \eqref{expand} gives that for any distinct vertices $i,j\in[n]$,
\begin{align}\label{Erijn}
\E\big[R_{ij}^{(n)}\big]=0.
\end{align}
Denote by $R_n$ the sum of all the remainder terms, i.e.,
\begin{align}\label{Rndef}
R_{n}=\sum_{i=1}^{n-1}\sum_{j=i+1}^nR_{ij}^{(n)}.
\end{align}
Then it follows by \eqref{Erijn} that $\E[R_n]=0$.
By \eqref{sijntijn} and the simple fact $I_{ik}\vee I_{jk}=I_{ik}+I_{jk}-I_{ik}I_{jk}$,  
we have that for any $1\le i\neq j\le n$,
\begin{align*}
(2-p)S_{ij}^{(n)}-pT_{ij}^{(n)}=\sum_{k\neq i,j}\big[(2-p)I_{ik}I_{jk}-pI_{ik}\vee I_{jk}\big]
=\sum_{k\neq i,j}\big[2I_{ik}I_{jk}-p\big(I_{ik}+I_{jk}\big)\big],  
\end{align*}
which, together with \eqref{Jndef} and \eqref{expand}, implies that
\begin{align}\label{Jndecom}
	J_{n}&=\frac{p}{2-p}+\frac{2}{n(n-1)}\sum_{i=1}^{n-1}\sum_{j=i+1}^n
	\left(\frac{(2-p)S_{ij}^{(n)}-pT_{ij}^{(n)}}{(n-2)p(2-p)^2}
	+R_{ij}^{(n)}\right)\notag\\
	&=\frac{p}{2-p}+\frac{2}{n(n-1)(n-2)(2-p)^{2}}\sum_{i=1}^{n-1}\sum_{j=i+1}^n\sum_{k\neq i,j}
	\Big(\frac{2I_{ik}I_{jk}}{p}-\big(I_{ik}+I_{jk}\big)\Big)+\frac{2}{n(n-1)}R_{n}\notag\\
	&=\frac{p}{2-p}-\frac{4P_{1,n}}{n(n-1)(2-p)^{2}}+\frac{4P_{2,n}}{n(n-1)(n-2)p(2-p)^{2}}
	+\frac{2}{n(n-1)}R_{n},
\end{align}
where
\begin{align*}
P_{1,n}=\sum_{i=1}^{n-1}\sum_{j=i+1}^nI_{ij}\quad \mbox{and} \quad
P_{2,n}=\sum_{i=1}^{n-1}\sum_{j=i+1}^n\sum_{k\neq i,j}I_{ij}I_{ik}
\end{align*}
denote the number of edges and the number of paths of length two in $G(n,p)$, respectively.
Further, we can rewrite \eqref{Jndecom} as
\begin{align}\label{jnnormalized}
\frac{(n-1)(2-p)^2}{\sqrt{8p(1-p)}}\Big(J_n-\frac{p}{2-p}\Big)=
\sqrt{\frac{2p}{1-p}}\left(-\frac{P_{1,n}}{np}+\frac{P_{2,n}}{n(n-2)p^2}\right)
+\frac{(2-p)^2}{n\sqrt{2p(1-p)}}R_n.
\end{align}

For $P_{1,n}$ and $P_{2,n}$, it is not hard to obtain that their expectations are given by
\begin{align*}
\E[P_{1,n}]=\frac12n(n-1)p, \quad \E[P_{2,n}]=\frac12n(n-1)(n-2)p^2.
\end{align*}
Applying Theorem 3(iii) in \cite{feng2013zagreb} yields that 
if $np\to\infty$ and $n^2(1-p)\to \infty$,
\begin{align*}
\sqrt{\frac{2p}{1-p}}\left(\frac{P_{1,n}-\frac12n(n-1)p}{np},
\frac{P_{2,n}-\frac12n(n-1)(n-2)p^2}{2n(n-2)p^2}\right)\convlaw(Z,Z),
\end{align*}
which implies that 
\begin{align}\label{p1np2nasynorm}
\sqrt{\frac{2p}{1-p}}\left(-\frac{P_{1,n}}{np}+\frac{P_{2,n}}{n(n-2)p^2}\right)
\convlaw Z.
\end{align}
To prove Theorem \ref{theomavejac}, 
by \eqref{jnnormalized} and \eqref{p1np2nasynorm} it is sufficient to show 
\begin{align*}
   \frac{R_n}{n\sqrt{p(1-p)}}\inprob 0. 
\end{align*}
That is, by Chebyshev's inequality and the fact $\E[R_n]=0$, we only need to prove 
\begin{align}\label{VarRn}
\V[R_n]=o\big(n^2p(1-p)\big).
\end{align}

On the other hand, by symmetry, it follows by \eqref{Rndef} that
\begin{align}\label{RVar}
	\V[R_{n}]&=\cov
	\Bigg(\sum_{i=1}^{n-1} \sum_{j=i+1}^n R_{ij}^{(n)},
	~\sum_{i=1}^{n-1} \sum_{j=i+1}^n R_{ij}^{(n)}\Bigg)\notag\\
	&=\frac{n(n-1)}{2}\cov\Bigg(R_{12}^{(n)},
	~\sum_{i=1}^{n-1} \sum_{j=i+1}^n R_{ij}^{(n)}\Bigg)\notag\\
	&=\frac12n(n-1)\V\big[R_{12}^{(n)}\big]+n(n-1)(n-2)\cov\big(R_{12}^{(n)},R_{13}^{(n)}\big)\notag\\
	&\quad+\frac14n(n-1)(n-2)(n-3)\cov\big(R_{12}^{(n)},R_{34}^{(n)}\big).
\end{align}
To prove \eqref{VarRn}, we next estimate the variance and covariances in \eqref{RVar} separately.
By the law of total expectation and \eqref{Erijn}, for the variance of $R_{12}^{(n)}$ we have
\begin{align}\label{ResVarR120}
\V\big[R_{12}^{(n)}\big]
&=\E\big[\big(R_{12}^{(n)}\big)^{2}\big]\notag\\
&=\E\big[\E\big[\big(R_{12}^{(n)}\big)^{2}\big|T_{12}^{(n)}\big]\big]\notag\\
&=\sum_{m=0}^{n-2}\E\big[\big(R_{12}^{(n)}\big)^{2}\big|T_{12}^{(n)}=m\big]\Prob(T_{12}^{(n)}=m).
\end{align}
Recalling that $T_{12}^{(n)}\sim {\rm Bin}(n-2,p(2-p))$, and $R_{12}^{(n)}=0$ if $T_{12}^{(n)}=0$. 
By \eqref{conditionalE}, \eqref{conditionalVar} and \eqref{res}, it follows that
\begin{align*}
\E\big[\big(R_{12}^{(n)}\big)^{2}\big|T_{12}^{(n)}=m\big]&=\frac1{(n-2)^2p^2(2-p)^2}\Big(\frac{(n-2)p(2-p)}{m}-1\Big)^2\V\big(S_{12}^{(n)}\big|T_{12}^{(n)}=m\big)\notag\\
&=\frac{2(1-p)}{(n-2)^2p(2-p)^4}\Big(\frac{(n-2)^2p^2(2-p)^2}{m}-2(n-2)p(2-p)+m\Big),
\end{align*}
which, together with \eqref{inverse} and \eqref{ResVarR120}, implies that
\begin{align}\label{ResVarR12}
\V\big[R_{12}^{(n)}\big]
&=\frac{2(1-p)}{(n-2)^2p(2-p)^4}\sum_{m=1}^{n-2}\Big(\frac{(n-2)^2p^2(2-p)^2}{m}-2(n-2)p(2-p)+m\Big)
\Prob\big(T_{12}^{(n)}=m\big)\notag\\
&=\frac{2(1-p)}{(n-2)(2-p)^3}\Big((n-2)p(2-p)\sum_{m=1}^{n-2}\frac1m\Prob(T_{12}^{(n)}=m)
-1+2\Prob\big(T_{12}^{(n)}=0\big)\Big)\notag\\
&=O\left(\frac{1-p}{n^2p}\right),
\end{align}
where in the last equality we used the simple fact
\begin{align*}
0<\Prob\big(T_{12}^{(n)}=0\big)=(1-p)^{2(n-2)}\leq {\rm e}^{-2(n-2)p}=o\left(\frac{1}{np}\right).
\end{align*}

To calculate the covariance of $R_{12}^{(n)}$ and $R_{13}^{(n)}$, 
by convention we introduce the following shorthand notation
\begin{align*}
    \widetilde{T}_{ij}^{(n)}:=\frac{(n-2)p(2-p)}{T_{ij}^{(n)}}-1, 
\end{align*}
for distinct vertices $i,j\in [n]$. Recalling \eqref{vijk}, we have 
\begin{align*}
    (2-p)S_{ij}^{(n)}-pT_{ij}^{(n)}=\sum_{k\neq i,j} V_{ij,k}, \quad i,j\in [n].
\end{align*}
By symmetry and (\ref{res}), it thus follows that
\begin{align} \label{CovR12R13}
    \cov\big(R_{12}^{(n)},R_{13}^{(n)}\big)
	&=\frac{1}{(n-2)^{2}p^{2}(2-p)^{4}}\sum_{k=3}^n\sum_{l\neq1,3}\E\big[V_{12,k}V_{13,l}\widetilde{T}_{12}^{(n)}\widetilde{T}_{13}^{(n)}
	{\rm I}\big(T_{12}^{(n)}T_{13}^{(n)}>0\big)\big]\notag\\
	&=\frac{1}{(n-2)p^{2}(2-p)^{4}}\E\big[V_{12,3}V_{13,2}\widetilde{T}_{12}^{(n)}\widetilde{T}_{13}^{(n)}{\rm I}\big(T_{12}^{(n)}T_{13}^{(n)}>0\big)\big]\notag\\
	&\quad+\frac{(n-3)}{(n-2)p^{2}(2-p)^{4}}\E\big[V_{12,3}V_{13,4}\widetilde{T}_{12}^{(n)}\widetilde{T}_{13}^{(n)}{\rm I}\big(T_{12}^{(n)}T_{13}^{(n)}>0\big)\big],
\end{align}
which splits the covariance into two parts.
Notice  that, by \eqref{vijk}, the discrete random variable $V_{ij,k}\neq 0$ if and only if $I_{ik}\vee I_{jk}=1$.
This implies that if $V_{12,3}V_{13,2}\neq 0$, we have  
\begin{align*}
   T_{12}^{(n)}=1+\sum_{k=4}^{n} I_{1k}\vee I_{2k}\quad\mbox{and}\quad T_{13}^{(n)}=1+\sum_{k=4}^{n} I_{1k}\vee I_{3k},
\end{align*}
and they have the same distribution. Since it follows that, by conditioning on $I_{23}$,  
\begin{align*}
\E[V_{12,3}V_{13,2}]&=\E\big[\big((2-p)I_{13}I_{23}-p(I_{13}\vee I_{23})\big)\big((2-p)I_{12}I_{23}-p(I_{12}\vee I_{23})\big)\big]\notag\\
                    &=p\E\big[\big((2-p)I_{13}-p\big)\big((2-p)I_{12}-p\big)\big]+(1-p)\E\big[p^2I_{13}I_{12}\big] \notag\\
                    &=p^3(1-p)^2+p^4(1-p)\notag\\
                    &=p^3(1-p),
\end{align*}
and that, by Lemma \ref{leminverse1} and the Cauchy-Schwarz inequality, 
\begin{align*}
    \left|\mathbb{E}\Big[\Big(\frac{(n-2)p(2-p)}{1+\sum_{k=4}^{n}I_{1k}\vee I_{2k}}-1\Big)\Big(\frac{(n-2)p(2-p)}{1+\sum_{k=4}^{n}I_{1k}\vee I_{3k}}-1\Big)\Big]\right|&\leq\mathbb{E}\Big[\Big(\frac{(n-2)p(2-p)}{1+\sum_{k=4}^{n}I_{1k}\vee I_{2k}}-1\Big)^{2}\Big]\notag\\
    &=O\Big(\frac{1}{np}\Big),
\end{align*}
we have
\begin{align}\label{V123V132}
&~\E\big[V_{12,3}V_{13,2}\widetilde{T}_{12}^{(n)}\widetilde{T}_{13}^{(n)}{\rm I}\big(T_{12}^{(n)}T_{13}^{(n)}>0\big)\big]\notag\\
=& ~\E\Big[V_{12,3}V_{13,2}\Big(\frac{(n-2)p(2-p)}{1+\sum_{k=4}^{n}I_{1k}\vee I_{2k}}-1\Big)
\Big(\frac{(n-2)p(2-p)}{1+\sum_{k=4}^{n}I_{1k}\vee I_{3k}}-1\Big)\Big]\notag\\
=& ~p^{3}(1-p)\mathbb{E}\Big[\Big(\frac{(n-2)p(2-p)}{1+\sum_{k=4}^{n}I_{1k}\vee I_{2k}}-1\Big)
\Big(\frac{(n-2)p(2-p)}{1+\sum_{k=4}^{n}I_{1k}\vee I_{3k}}-1\Big)\Big] \notag\\
=& ~O\Big(\frac{p^{2}(1-p)}{n}\Big).
\end{align}
Let us define
\begin{align*}
{\cal B}_{ab}:=\big\{I_{12}\vee I_{23}=a,~ I_{14}\vee I_{24}=b\big\},  \quad a,b \in\{0,1\}. 
\end{align*}
Noting that $\mathbb{E}[V_{12,3}]=0$, we have
\begin{align*}
\mathbb{E}\big[V_{12,3}{\rm I}(I_{12}\vee I_{23}=1)\big]
&=-\mathbb{E}\big[V_{12,3}{\rm I}(I_{12}\vee I_{23}=0)\big]\\
&=p\mathbb{E}\big[I_{13}{\rm I}(I_{12}=I_{23}=0)\big]\\
&=p^2(1-p)^2,
\end{align*}
which implies that for any $a,b=0$ or 1,
\begin{align*}
\mathbb{E}\big[V_{12,3}V_{13,4}{\rm I}({\cal B}_{ab})\big]
&=\mathbb{E}\big[V_{12,3}{\rm I}(I_{12}\vee I_{23}=a)\big]
\mathbb{E}\big[V_{13,4}{\rm I}(I_{14}\vee I_{24}=b)\big]\\
&=\mathbb{E}\big[V_{12,3}{\rm I}(I_{12}\vee I_{23}=a)\big]
\mathbb{E}\big[V_{12,3}{\rm I}(I_{12}\vee I_{23}=b)\big]\\
&=(-1)^{a+b}p^4(1-p)^4.
\end{align*}
Analogously to \eqref{V123V132}, we have 
\begin{align}\label{EV123V134}
	&\quad \mathbb{E}\big[V_{12,3}V_{13,4}\widetilde{T}_{12}^{(n)}\widetilde{T}_{13}^{(n)}{\rm I}\big(T_{12}^{(n)}T_{13}^{(n)}>0\big)\big]\notag\\
	&=\sum_{a=0}^1\sum_{b=0}^1\mathbb{E}\big[V_{12,3}V_{13,4}\widetilde{T}_{12}^{(n)}\widetilde{T}_{13}^{(n)}{\rm I}({\cal B}_{ab})\big]\notag\\
	&=\sum_{a=0}^1\sum_{b=0}^1
	\mathbb{E}\big[V_{12,3}V_{13,4}W_{12}(b)W_{13}(a){\rm I}({\cal B}_{ab})\big]\notag\\
	&=\sum_{a=0}^1\sum_{b=0}^1 	\mathbb{E}\big[V_{12,3}V_{13,4}{\rm I}({\cal B}_{ab})\big]
	\mathbb{E}\big[W_{12}(b)W_{13}(a)\big]\notag\\
	&=p^{4}(1-p)^{4}\mathbb{E}\left[\left(W_{12}(0)-W_{12}(1)\right)\left(W_{13}(0)-W_{13}(1)\right)\right],
\end{align}
where
\begin{align*}
W_{ij}(a):=\frac{(n-2)p(2-p)}{1+a+\sum_{k=5}^{n}I_{ik}\vee I_{jk}}-1, \quad i,j\in [n], ~a=0,1. 
\end{align*}
Noting that $W_{12}(0)-W_{12}(1)$ and $W_{13}(0)-W_{13}(1)$ have the same distribution, 
by Lemma \ref{leminverse2} and the Cauchy-Schwarz inequality, we have
\begin{align*}
	\E \big|\big(W_{12}(0)-W_{12}(1)\big)\big(W_{13}(0)-W_{13}(1)\big)\big|
	&\le \E \big[\big(W_{12}(0)-W_{12}(1)\big)^2\big]\\
	&\le (n-2)^{2}p^{2}(2-p)^{2}\E\left[\frac{1}{\big(1+\sum_{k=5}^{n}I_{1k}\vee I_{2k}\big)^{4}}\right]\\
	&=O\left(\frac{1}{n^{2}p^{2}}\right),
\end{align*}
which, together with (\ref{CovR12R13}), (\ref{V123V132}) and (\ref{EV123V134}), implies that 
\begin{align}\label{ResuCovR12R13}
\cov\big(R_{12}^{(n)},R_{13}^{(n)}\big)=O\Big(\frac{1-p}{n^2}\Big).
\end{align}

It remains to calculate the second covariance $\mathbf{Cov}(R_{12}^{(n)},R_{34}^{(n)})$ on the right-hand side of \eqref{RVar}, 
and the procedure is similar to the previous one. 
We shall sketch the calculations below, and omit a few specific interpretations. 
Analogously to (\ref{CovR12R13}), we have
\begin{align}\label{ER12R34}
\cov\big(R_{12}^{(n)},R_{34}^{(n)}\big)&=\frac{1}{(n-2)^{2}p^{2}(2-p)^4}
  \sum_{k=3}^n\sum_{l\neq 3,4}\E\big[V_{12,k}V_{34,l}\widetilde{T}_{12}^{(n)}\widetilde{T}_{34}^{(n)}
    {\rm I}\big(T_{12}^{(n)}T_{34}^{(n)}>0\big)\big]\notag\\
 &=\frac{4}{(n-2)^{2}p^{2}(2-p)^4}\E\big[V_{12,3}V_{34,1}\widetilde{T}_{12}^{(n)}\widetilde{T}_{34}^{(n)}
    {\rm I}\big(T_{12}^{(n)}T_{34}^{(n)}>0\big)\big]\notag\\
 &\quad+\frac{2}{(n-2)p^{2}(2-p)^4}\E\big[V_{12,3}V_{34,5}\widetilde{T}_{12}^{(n)}\widetilde{T}_{34}^{(n)}
    {\rm I}\big(T_{12}^{(n)}T_{34}^{(n)}>0\big)\big]\notag\\  
 &\quad+\frac{(n-4)^2}{(n-2)^{2}p^{2}(2-p)^4}\E\big[V_{12,5}V_{34,5}\widetilde{T}_{12}^{(n)}\widetilde{T}_{34}^{(n)}
    {\rm I}\big(T_{12}^{(n)}T_{34}^{(n)}>0\big)\big].   
\end{align}
Define
\begin{align*}
{\cal C}_{ab}:=\big\{I_{14}\vee I_{24}=a,~ I_{23}\vee I_{24}=b\big\},  \quad a,b \in\{0,1\}. 
\end{align*}
After straightforward calculations, we have
\begin{align*}
\E\big[V_{12,3}V_{34,1}{\rm I}({\cal C}_{00})\big]&=p^3(1-p)^3,\\
\E\big[V_{12,3}V_{34,1}{\rm I}({\cal C}_{11})\big]&=p^{3}(1-p)\big(1+3(1-p)^{2}\big),
\end{align*}
and
\begin{align*}
\E\big[V_{12,3}V_{34,1}{\rm I}({\cal C}_{01})\big]=\E\big[V_{12,3}V_{34,1}{\rm I}({\cal C}_{10})\big]
= -2p^3(1-p)^3.
\end{align*}
Then one can conclude that
\begin{align*}
  \E\big[V_{12,3}V_{34,1}{\rm I}({\cal C}_{ab})\big]=O\big(p^3(1-p)\big).  
\end{align*}
Hence, by Cauchy-Schwarz inequality and Lemma \ref{leminverse1},
\begin{align}\label{EV123V341}
\E\big[V_{12,3}V_{34,1}\widetilde{T}_{12}^{(n)}\widetilde{T}_{34}^{(n)}{\rm I}\big(T_{12}^{(n)}T_{34}^{(n)}>0\big)\big]
&=\sum_{a=0}^1\sum_{b=0}^1\E\big[V_{12,3}V_{34,1}W_{12}(a)W_{34}(b){\rm I}({\cal C}_{ab})\big]\notag\\
&=\sum_{a=0}^1\sum_{b=0}^1\E\big[V_{12,3}V_{34,1}{\rm I}({\cal C}_{ab})\big]\E[W_{12}(a)]\E[W_{34}(b)]\notag\\
&=O\Big(\frac{p^2(1-p)}{n}\Big).
\end{align}
Since $V_{12,3}$ and $V_{34,5}$ are independent and have the common mean 0, 
it follows that
\begin{align}\label{EV123V345}
&\quad\E\big[V_{12,3}V_{34,5}\widetilde{T}_{12}^{(n)}\widetilde{T}_{34}^{(n)}
    {\rm I}\big(T_{12}^{(n)}T_{34}^{(n)}>0\big)\big]\notag\\
    &=\E\big[V_{12,3}V_{34,5}\widetilde{T}_{12}^{(n)}\widetilde{T}_{34}^{(n)}
    {\rm I}\big(I_{13}\vee I_{23}=1\big){\rm I}\big(I_{35}\vee I_{45}=1\big)\big]\notag\\
    &=\E\big[V_{34,5}{\rm I}\big(I_{35}\vee I_{45}=1\big)\big]
    \E\Big[V_{12,3}\widetilde{T}_{12}^{(n)}{\rm I}\big(I_{13}\vee I_{23}=1\big)
    \Big(\frac{(n-2)p(2-p)}{1+\sum_{k\neq 3,4,5}(I_{3k}\vee I_{4k})}-1\Big)\Big]\notag\\
    &=\E\big[V_{34,5}\big]
    \E\Big[V_{12,3}\widetilde{T}_{12}^{(n)}{\rm I}\big(I_{13}\vee I_{23}=1\big)
    \Big(\frac{(n-2)p(2-p)}{1+\sum_{k\neq 3,4,5}(I_{3k}\vee I_{4k})}-1\Big)\Big]\notag\\
    &=0.
\end{align}
Similarly, we also have
\begin{align}\label{EV125V345}
\E\big[V_{12,5}V_{34,5}\widetilde{T}_{12}^{(n)}\widetilde{T}_{34}^{(n)}
    {\rm I}\big(T_{12}^{(n)}T_{34}^{(n)}>0\big)\big]=0.
\end{align}
Plugging \eqref{EV123V341}, \eqref{EV123V345} and \eqref{EV125V345} 
into \eqref{ER12R34} yields that 
\begin{align*}
  \cov\big(R_{12}^{(n)},R_{34}^{(n)}\big)=O\Big(\frac{1-p}{n^3}\Big),  
\end{align*}
which, together with (\ref{RVar}), (\ref{ResVarR12}) and (\ref{ResuCovR12R13}),
implies that 
\begin{align*}
{\bf Var}[R_n]=O\Big(\frac{1-p}{p}\Big)+O\big(n(1-p)\big)+O\big(n(1-p)\big)=O\big(n(1-p)\big),
\end{align*}
as $np\to\infty$.
This proves \eqref{VarRn}, and thus completes the proof of Theorem \ref{theomavejac}. 
\end{proof}


\begin{thebibliography}{99}

	
	\bibitem{ali2021machine}
	{\sc Ali, M., Jiang, R.,  Ma, H., Pan, H., Abbas, K., Ashraf, U. and  Ullah, J.} (2021).
	Machine learning-a novel approach of well logs similarity based on synchronization measures to predict shear sonic
	logs. {\em Journal of Petroleum Science and Engineering} {\bf 203}, 108602.
	
	\bibitem{arias2014community}
	{\sc Arias-Castro, E. and Verzelen, N.} (2014). Community detection in dense random networks. {\em The
		Annals of Statistics} {\bf 42} (3), 940--969.
	
	\bibitem{bag2019efficient}
	{\sc Bag, S.,  Kumar, S. K. and Tiwari, M. K.}  (2019). An efficient recommendation generation using
	relevant Jaccard similarity. {\em Information Sciences}  {\bf 483}, 53--64.
	
	\bibitem{berahmand2018community}
	{\sc Berahmand, K., Bouyer, A. and Vasighi, M.} (2018). Community detection in complex networks by detecting and expanding core nodes through extended local similarity of nodes. {\em IEEE
		Transactions on Computational Social Systems} {\bf 5}(4), 1021--1033.
	
	\bibitem{bollobas2001random}
	{\sc Bollob{\'a}s, B.} (2001). {\em Random Graphs}, 2nd Edition. Cambridge: Cambridge University Press.
	
	
	\bibitem{chung2019jaccard}
	{\sc Chung, N. C.,  Miasojedow, B.,  Startek, M. and Gambin, A.} (2019). Jaccard/Tanimoto similarity
	test and estimation methods for biological presence-absence data. {\em BMC bioinformatics} {\bf 20} (15),
	1--11.
	
	\bibitem{costa2021further}
	{\sc da Fontoura Costa, L.} (2021). Further generalizations of the Jaccard index. arXiv preprint
	arXiv:2110.09619.
	
	
	\bibitem{eelbode2020optim}
	{\sc Eelbode, T.,  Bertels, J., Berman, M.,  Vandermeulen, D.,  Maes, F., Bisschops,R. and Blaschko,  M. B.}
	(2020). Optimization for medical image segmentation: Theory and practice when evaluating
	with dice score or Jaccard index. {\em IEEE Transactions on Medical Imaging} {\bf 39}(11), 3679--3690.
	
	\bibitem{fan2019similarity}
	{\sc Fan, X., Li, X., Yin, J., Tian, L. and  Liang, J.} (2019). Similarity and heterogeneity of price
	dynamics across Chinas regional carbon markets: A visibility graph network approach. {\em Applied
		Energy} {\bf 235}, 739--746.
	
	
	\bibitem{feng2013zagreb}
	{\sc Feng, Q.,  Hu, Z. and  Su, C.} (2013). The Zagreb indices of random graphs. {\em Probability in the
		Engineering and Informational Sciences} {\bf 27}(2), 247--260.
	
	
	\bibitem{gilbert1972distance}
	{\sc Gilbert, G.} (1972). Distance between sets. {\em Nature} {\bf 239}, 174.
	
	\bibitem{henning2007cluster}
	{\sc Hennig, C.} (2007). Cluster-wise assessment of cluster stability. {\em Computational Statistics and Data
		Analysis} {\bf 52}(1), 258--271.
	
	
	\bibitem{jaccard1901distribution}
	{\sc Jaccard, P.} (1901). Distribution de la flore alpine dans le bassin des dranses et dans quelques
	r{\'e}gions voisines. {\it Bulletin de la Soci\'{e}t\'{e}  Vaudoise des Sciences Naturelles} {\bf 37}, 241--272.
	
	
	\bibitem{janson2000random}
	{\sc Janson, S.,  Luczak, T. and  Rucinski, A.} (2000).  {\em Random graphs.} New York: John Wiley \& Sons.
	
	
	\bibitem{koeneman2022improved}
	{\sc Koeneman, S. H. and Cavanaugh, J. E.} (2022). An improved asymptotic test for the Jaccard
	similarity index for binary data. {\em Statistics and Probability Letters} {\bf 184}, 109375.
	
	\bibitem{kogge2016jaccard}
	{\sc Kogge, P. M.} (2016). Jaccard coefficients as a potential graph benchmark. In {\em 2016 IEEE International Parallel and Distributed Processing Symposium Workshops (IPDPSW)}, pp. 921--928.
	IEEE.
	
	\bibitem{kosub2019note}
	{\sc Kosub, S.} (2019). A note on the triangle inequality for the Jaccard distance. {\em Pattern Recognition
		Letters} {\bf 120}, 36--38.
	
	\bibitem{lu2023embedding}
	{\sc Lu, H. and Uddin, S.} (2023). Embedding-based link predictions to explore latent comorbidity of
	chronic diseases. {\em Health Information Science and Systems} {\bf 11}, 2.
	
	\bibitem{mammone2018permutation}
	{\sc Mammone, N., Ieracitano, C., Adeli, H., Bramanti, A. and  Morabito, F. C.} (2018). Permutation
	jaccard distance-based hierarchical clustering to estimate eeg network density modifications in
	mci subjects. {\em IEEE transactions on neural networks and learning systems} {\bf 29} (10), 5122--5135.
	
	\bibitem{miasnikof2022empirical}
	{\sc Miasnikof, P.,  Shestopaloff, A. Y.,  Pitsoulis, L. and Ponomarenko, A.} (2022). An empirical comparison of connectivity-based distances on a graph and their computational scalability. {\em Journal
		of Complex Networks} {\bf 10}(1), cnac003.
	
	
	\bibitem{sathre2022edge}
	{\sc Sathre, P.,  Gondhalekar, A. and Feng, W. C.} (2022). Edge-connected Jaccard similarity for graph
	link prediction on FPGA. In {\em  2022 IEEE High Performance Extreme Computing Conference
		(HPEC)}, pp. 1--10. IEEE.
	
	\bibitem{shestopaloff2022statistical}
	{\sc Shestopaloff, P. M., Alexander,Y.,  Bravo, C. and Lawryshyn, Y.} (2023). Statistical network isomorphism. In H. Cherifi, R. N. Mantegna, L. M. Rocha, C. Cherifi, and S. Micciche (Eds.),
	{\em Complex Networks and Their Applications XI,} pp. 325--336. Springer International Publishing.
	
	
	\bibitem{shi2010note}
	{\sc Shi, X.,  Wu, Y. and  Liu, Y.} (2010). A note on asymptotic approximations of inverse moments of
	nonnegative random variables. {\em Statistics and Probability Letters} {\bf 80}(15-16), 1260--1264.
	
	\bibitem{singh2009privacy}
	{\sc Singh, M. D.,  Krishna, P. R. and  Saxena, A.} (2009). A privacy preserving Jaccard similarity
	function for mining encrypted data. In {\em TENCON 2009 - 2009 IEEE Region 10 Conference,} pp.
	1--4.
	
	\bibitem{hofstad2016random}
	{\sc van der Hofstad, R.} (2016). {\em Random Graphs and Complex Networks.} Cambridge: Cambridge
	University Press.
	
	
	\bibitem{verzelen2015community}
	{\sc Verzelen, N. and  Arias-Castro, E.} (2015). Community detection in sparse random networks. {\em The
		Annals of Applied Probability} {\bf 25}(6), 3465--3510.
	
	
	\bibitem{wu2017extractin}
	{\sc Wu, C. and  Wang, B.} (2017). Extracting topics based on word2vec and improved Jaccard similarity
	coefficient. In {\em 2017 IEEE Second International Conference on Data Science in Cyberspace
		(DSC),} pp. 389--397.
	
	\bibitem{wang2008asymptotic}
	{\sc Wuyungaowa and  Wang, T. }(2008). Asymptotic expansions for inverse moments of binomial and
	negative binomial.{\em  Statistics and Probability Letters} {\bf 78} (17), 3018--3022.
	
	\bibitem{yin2006similarity}
	{\sc Yin, Y. and Yasuda, K.} (2006). Similarity coefficient methods applied to the cell formation problem:
	A taxonomy and review. {\em International Journal of Production Economics} {\bf 101} (2), 329--352.
	
	\bibitem{zhang2016measuring}
	{\sc Zhang, P., Wang, X., Wang, F., Zeng, A. and  Xiao, J.} (2016). Measuring the robustness of link
	prediction algorithms under noisy environment. {\em Scientific Reports} {\bf 6}, 18881.
	
	
	
	
\end{thebibliography}
\end{document}